\documentclass[reqno,a4paper,12pt]{amsart}

\usepackage{amsmath,amscd,amsfonts,amssymb}
\usepackage{mathrsfs,dsfont}

\numberwithin{equation}{section}

\newcommand{\define}[1]{{\emph{#1}}}

\def\R{\mathbb{R}}

\def\Z{\mathbb{Z}}

\def\B{\mathfrak{B}}

\def\gam{\gamma}
\def\Gam{\Gamma}
\def\lam{\lambda}
\def\Lam{\Lambda}
\def\1{\mathds{1}}

\renewcommand\le{\leqslant}
\renewcommand\ge{\geqslant}
\renewcommand\leq{\leqslant}
\renewcommand\geq{\geqslant}

\renewcommand\hat{\widehat}

\newcommand{\ft}[1]{\widehat #1}
\newcommand{\dotprod}[2]{\langle #1, #2  \rangle }
\newcommand\mes{\operatorname{mes}}

\newcommand{\supp}{\operatorname{supp}}
\newcommand{\spec}{\operatorname{spec}}

\newcommand{\dist}{\operatorname{dist}}

\theoremstyle{plain}
\newtheorem{lem}{Lemma}[section]
\newtheorem{thm}[lem]{Theorem}
\newtheorem{corollary}[lem]{Corollary}
\newtheorem{prop}[lem]{Proposition}

\newcommand{\thmref}[1]{Theorem~\ref{#1}}

\newcommand{\lemref}[1]{Lemma~\ref{#1}}
\newcommand{\propref}[1]{Proposition~\ref{#1}}
\newcommand{\corref}[1]{Corollary~\ref{#1}}

\newtheorem*{theorem-m}{Theorem M}

\theoremstyle{definition}

\newtheorem*{definition*}{Definition}

\newtheorem*{remark}{Remark}

\newenvironment{enumerate-math}
{\begin{enumerate}
\addtolength{\itemsep}{5pt}
}
{\end{enumerate}}

\newenvironment{enumerate-math-abc}
{\begin{enumerate}
\addtolength{\itemsep}{5pt}
}
{\end{enumerate}}

\begin{document}

\title[Fourier quasicrystals]{Fourier quasicrystals and discreteness of the diffraction spectrum}
\author{Nir Lev}
\address{Department of Mathematics, Bar-Ilan University, Ramat-Gan 52900, Israel}
\email{levnir@math.biu.ac.il}

\author{Alexander Olevskii}
\address{School of Mathematical Sciences, Tel-Aviv University, Tel-Aviv 69978, Israel}
\email{olevskii@post.tau.ac.il}

\date{May 18, 2017}
\thanks{Research supported by ISF grants No.\ 225/13 and 455/15 and ERC Starting Grant No.\ 713927.}

\begin{abstract}
We prove that  a positive-definite measure in $\R^n$ 
with uniformly discrete support and discrete closed  spectrum,
is representable as a finite linear combination of Dirac 
combs,  translated and modulated. 
This extends our recent results where we proved this
under the assumption that also the spectrum is uniformly discrete.
As an application we obtain that Hof's quasicrystals with
uniformly discrete diffraction spectra must have a periodic diffraction structure.

\end{abstract}

\maketitle

% ========================================================================

\section{Introduction}

\subsection{}
       By a \emph{Fourier quasicrystal} one often means a 
discrete measure,  whose Fourier transform  is also a 
discrete measure.
       This concept was inspired by the experimental
       discovery in the middle of 80's of non-periodic 
       atomic structures with diffraction patterns consisting
       of spots. In this context, different versions of ``discreteness'' were discussed,
              see in particular \cite{bom}, \cite{cah},   \cite{mey2}, \cite{dys}.
\par
 Let $\mu$ be a (complex) measure on $\R^n$, supported on a discrete set $\Lambda$: 
 \begin{equation}
	 \mu =\sum_{\lambda\in \Lambda} \mu(\lambda)\delta_{\lambda}, \quad
	 \mu(\lambda)\ne 0.
	 \label{RM.1}
 \end{equation}
We shall suppose that $\mu$ is a slowly increasing measure, which means that
$|\mu|\{x:|x|<r\}$ grows at most polynomially as $r\to\infty$. 
Hence the Fourier transform 
 \[ \hat{\mu}(t):= \sum_{\lambda\in\Lambda} \mu(\lambda)e^{-2\pi i\dotprod{\lambda}{t}} \]
is  well-defined as a temperate distribution on $\R^n$. Assume that also $\ft\mu$ is a 
slowly increasing measure, which is purely atomic and  supported on a countable  set $S$: 
 \begin{equation}
	 \hat{\mu}=\sum_{s\in S} \hat{\mu}(s)\delta_{s}, \quad \hat{\mu}(s)\ne 0.
	 \label{RM.2}
 \end{equation}
 The set $\Lambda$ is then called the \emph{support} of the measure $\mu$, while $S$ is called the \emph{spectrum}.

\subsection{}
In \cite[Problem 4.1(a)]{lag2} the following question was posed:
\par
\emph{Suppose that $\mu$ is a positive measure, whose support $\Lam$ and spectrum $S$ are both   uniformly discrete sets.
           Is it true that $\Lam$ can be covered by a finite union of translates
          of a certain lattice?}
\par
Recall that a set is  uniformly discrete if the distance between 
any two of its points is bounded  below by some positive constant.
\par
In \cite{lo1, lo2} we proved the following result, which answers the above question affirmatively,
and moreover shows that the measure $\mu$ must have a periodic structure:
 \begin{thm}
	 \label{RT.0}
	 Let $\mu$ be a positive (or positive-definite) measure on $\R^n$ satisfying  \eqref{RM.1} and
	 \eqref{RM.2}, and assume that $\Lambda$ and $S$ are both uniformly discrete sets.
	 Then there is a lattice $L$, such that $\mu$ is representable  in the form
	 \begin{equation}
		 \mu=\sum_{j=1}^{N}\sum_{\lambda\in L+\tau_j}P_j(\lambda) \, \delta_\lambda
		 \label{RM.3}
	 \end{equation}
	 for certain vectors $\tau_j$ in $\R^n$ and  trigonometric polynomials $P_j$ $(1 \leq j \leq N)$. 
Moreover, for $n=1$ the result holds even without the
 positivity assumption on the measure $\mu$. 
 \end{thm}
           So the theorem says that under the conditions above,
          the measure $\mu$ is a finite linear combination of Dirac combs,  translated and
          modulated. Conversely, for any measure $\mu$ of the form \eqref{RM.3}, the
	support $\Lam$ and spectrum $S$ are both  uniformly discrete  sets.
\par
           It was also an open problem (see \cite[Problem 4.1(b)]{lag2})
           whether such a result can be proved in the more general
           situation when $\Lambda$, $S$ are assumed to be just discrete 
           closed sets. This problem was addressed in \cite{lo3} where  we proved the following:
 \begin{thm}
	 \label{RT.4}
	 There is a (non-zero)  measure $\mu$ on $\R$ satisfying
	 \eqref{RM.1} and \eqref{RM.2}, such that $\Lambda$ and $S$ are both discrete closed
	 sets, but $\Lambda$ contains only finitely many elements of any arithmetic
	 progression.
 \end{thm}
\par
 The latter condition indicates that the measure $\mu$ is ``non-periodic'' in a strong
 sense. In particular, it cannot be obtained as  a finite linear combination of Dirac combs.
 A result of similar type is true also in $\R^n,$ $n>1$.
\par
Moreover, in our construction both  $\mu$ and $\ft\mu$ 
are translation-bounded measures (see Section \ref{sec:prelim} for the
 definition). This implies that 
$\mu$ is an almost periodic measure, whose Fourier
transform $\ft\mu$ is also an almost periodic measure.
By definition, a measure $\mu$ on $\R^n$ is an \define{almost periodic
measure} if for every continuous, compactly supported function
$\varphi$ on $\R^n$, the convolution $\mu \ast \varphi$ is an
almost periodic function in the sense of Bohr.

\subsection{}
In the crystallography community, it seems to be commonly agreed
 that the support  $\Lambda$ should be a uniformly discrete set.
    We remind that Meyer's quasicrystals \cite{mey2},
               which appeared first  under the name ``model sets'' \cite{mey0},
               are uniformly discrete sets, and they support measures
               whose spectra are dense countable sets.
\par
               So it is a natural problem, to  what extent can  the spectrum $S$ of a non-periodic 
quasi\-crystal be discrete, assuming that the support  $\Lambda$ is uniformly discrete?
 In the present paper we address this problem, and  consider quasicrystals with non-symmetric
discreteness assumptions on the support  and  the spectrum.
\par
First we obtain several results which show that, under certain conditions,
if  the spectrum $S$ is a discrete closed set, then in fact $S$  must    be uniformly discrete.
These results thus reduce the situation to the setting in \thmref{RT.0} above, which in turn
allows us to conclude that the measure $\mu$  is representable  in the form
\eqref{RM.3}.
\par
On the other hand, we present an example of a  non-periodic quasicrystal such
that the spectrum $S$ is a nowhere dense countable set.
\par
We also apply our results to Hof's quasicrystals. In this context, we prove that
if a Delone set $\Lam$ of finite local complexity has a uniformly discrete diffraction spectrum,
then the diffraction measure of $\Lam$ has a periodic structure.
\par
Finally, we extend our results to the more general situation, where $\hat{\mu}$
 is a measure which has both a pure point component and a continuous one.

% ====================================================================  

\section{Results}

\subsection{}
 Our first result deals with  the case when $\mu$ is a  positive-definite measure: 
 \begin{thm}
	 \label{RT.5}
	 Let $\mu$ be a positive-definite measure on $\R^n$
	 satisfying \eqref{RM.1} and \eqref{RM.2}. Assume that the support $\Lambda$ is a
	 uniformly discrete set, while the spectrum $S$ is a discrete closed set. Then
	 also $S$ is uniformly discrete, and the measure $\mu$ has the form \eqref{RM.3}. 
 \end{thm}
	 Actually, we will show that if a positive-definite
	 measure $\mu$ with uniformly discrete support $\Lambda$ is not of the form
	 \eqref{RM.3}, then its spectrum $S$ must have a relatively dense set of
	 accumulation points   (see \thmref{AT.5}). 

\subsection{}
In the next result, which holds for dimension $n=1$, the
	 positive-definiteness assumption is replaced by a stronger discreteness condition
	 on the spectrum:
\begin{thm}
	\label{RT.6}
	Let $\mu$ be a measure on $\R$ satisfying \eqref{RM.1} and \eqref{RM.2}. Assume
	that  the support $\Lambda$ is a uniformly discrete set, while
 the spectrum $S$ satisfies the condition
	 \begin{equation}
 \sup_{x\in\R}\# (S \cap [x,x+1]) < \infty. 
		 \label{RM.4}
	 \end{equation}
Then	$S$ is a uniformly discrete set, and $\mu$ is of the form \eqref{RM.3}. 
\end{thm}
Notice that  condition \eqref{RM.4} means that   $S$ is the union of a finite number of uniformly discrete sets. 

\subsection{}
In the following result, a stronger discreteness condition is imposed on the support:
\begin{thm}
	\label{RT.7}
	Let $\mu$ be a measure on $\R^n$ satisfying \eqref{RM.1} and
	\eqref{RM.2}.
	Assume that the set $\Lambda-\Lambda$ is  uniformly discrete, and that
	$S$ is a discrete closed set. Then the same conclusion as in the previous two theorems holds. 
\end{thm}
	 Moreover, we will prove that if $\Lambda-\Lambda$ is  uniformly discrete, but the
	 measure $\mu$ is not of the form
	 \eqref{RM.3}, then again the spectrum $S$ must have a relatively dense set of
	 accumulation points (see \thmref{CT.1}).
\par
\thmref{RT.7} implies the following result communicated to us by Y.\ Meyer:
a ``simple quasicrystal'' cannot support a measure whose spectrum
is a discrete closed set (for the definition of a simple quasicrystal, see \cite{matei-meyer-simple}).

\subsection{}
 The previous results show that if the measure $\mu$
does not have the periodic structure \eqref{RM.3},
  then  the spectrum $S$ must have finite accumulation points.
   However, $S$ need not be dense in any ball, as the following result shows:
\begin{thm}
	\label{RT.15}
	There is a  positive-definite measure $\mu$ on $\R^n$
satisfying \eqref{RM.1} and \eqref{RM.2}, such that:
\begin{enumerate-math}
\item
The  support $\Lam$ is not contained in a finite union of translates of any lattice;
\item
The set $\Lam-\Lam$ is uniformly discrete; 
\item
The spectrum $S$ is a  nowhere dense (countable) set.
\end{enumerate-math}
\end{thm}
To prove this we base on Meyer's construction,
but with an additional modification.

\subsection{}
Our results can be applied to quasicrystals in Hof's sense \cite{hof}.
\par
Recall that a set $\Lambda\subset\R^n$ is called a \emph{Delone set} if it is  uniformly discrete and also
relatively dense. A Delone set $\Lambda$ is said to be of \emph{finite local complexity} if the difference
set $\Lambda-\Lambda$ is a discrete closed set. This means that  $\Lambda$ has, up to translations,
only a finite number of local patterns of any given size. 
\par
An \emph{autocorrelation measure} $\gamma_\Lam$ of the set $\Lambda$ is any weak limit point of the
measures 
	 \begin{equation}
 (2R)^{-n}\sum_{\lambda,\lambda'\in \Lam \cap [-R,R]^{n}}\delta_{\lambda'-\lambda}
		 \label{RT.8.1}
	 \end{equation}
as $R\to\infty$.  An autocorrelation measure $\gamma_\Lam$ is always
positive-definite, and so its Fourier transform $\hat{\gamma}_\Lam$ is a positive
measure, called a \emph{diffraction measure} of $\Lambda$.   If the measure
$\hat{\gamma}_\Lambda$ is purely atomic, then 
its support $S$ is called a
\define{diffraction spectrum} of $\Lambda$.
\par
The following result answers a question posed in \cite[Problem 4.2(a)]{lag2}:
\begin{thm}
	\label{RT.8}
Suppose that
\begin{enumerate-math}
\item
 $\Lambda\subset\R^n$ is a Delone set of finite local complexity;
\item
	$\gamma_\Lam$ is an autocorrelation measure of $\Lambda$; and
\item
	The diffraction spectrum $S$ (the support of $\hat{\gamma}_\Lam$) is a uniformly discrete set. 
\end{enumerate-math}
Then  $S$ is contained in a finite union of translates of some lattice,
and the diffraction measure $\hat{\gamma}_\Lambda$ has the form \eqref{RM.3}.
\end{thm}
A slightly more general version of this result will be given in \thmref{DT.1}.
The same conclusion is true if the set $\Lambda-\Lambda$ is uniformly discrete, 
and the diffraction spectrum $S$ is a discrete closed set (see \thmref{DT.2}).

\subsection{}
We also consider discrete measures $\mu$, whose
Fourier transform $\hat{\mu}$ is a measure which has both a pure point
component and a continuous one. We can extend our previous results
to this more general situation using the following:
\begin{thm}
	\label{RT.10} 
	Let $\mu$ be a measure on $\R^n$ with uniformly discrete
support $\Lambda$, and assume that  $\hat{\mu}$ is a (slowly increasing)
measure. Then the discrete part of $\hat{\mu}$ is the Fourier transform
of another measure $\mu'$, whose support $\Lambda'$ is also a uniformly discrete set.
\end{thm}
Moreover, if $\Lambda-\Lambda$ is a uniformly discrete set,
then also $\Lambda'-\Lambda'$ is uniformly discrete.
\par
By applying the previous results to this new measure $\mu'$, one can obtain
versions of the results for measures $\mu$ with non pure point Fourier transform
(see \thmref{FT.2}).

% ========================================================================

\section{Preliminaries}
\label{sec:prelim}

\subsection{Notation}
By $\dotprod{\cdot}{\cdot}$ and $|\cdot|$ we  denote the Euclidean 
scalar product and norm in $\R^n$. 
The open ball of radius $r$ centered at the origin is denoted
 $B_r := \{x \in \R^n: |x|<r \}$.
\par
A set $\Lambda \subset \R^n$ is said to be a \emph{discrete closed set} if $\Lambda$ has only finitely many 
points in any bounded set. The set $\Lambda$ is called \emph{uniformly discrete} if
\begin{equation}
\label{uddef}
                d(\Lam):=\inf_{\lam,\lam'\in\Lam, \lam\neq\lam'} |\lam-\lam'| > 0.
\end{equation}
The set $\Lambda$ is said to be  \emph{relatively dense}  if there is $R > 0$ such that every ball of
radius $R$ intersects $\Lam$. 
\par
By  a (full-rank) lattice $L \subset \R^n$ we mean the image of $\Z^n$ under some
invertible linear transformation $T$. The determinant $\det(L)$ is equal to $|\det (T)|$.
The dual lattice $L^*$ is the set of all vectors $\lambda^*$ such that $\dotprod{\lambda}{\lambda^*}
 \in \Z$, $\lambda \in L$.
\par
By a ``distribution'' we  will mean a temperate distribution on $\R^n$.
By a  ``measure'' we mean a complex, locally finite measure (usually infinite) which is
assumed to be \emph{slowly increasing}. By definition, a measure $\mu$ is slowly
increasing if  there is a constant $N$ such that 
$|\mu|(B_R)=O(R^N)$ as $R\to \infty.$
The measure $\mu$  is called \emph{translation-bounded} if
\begin{equation}\label{tr-bdd-def}
\sup_{x \in \R^n} |\mu|(x+B_1) < \infty.
\end{equation}
\par
Any translation-bounded measure  is slowly increasing, and any
slowly increasing measure  is a temperate distribution. Remark
that for a \emph{positive} measure to be a  temperate distribution, it is 
also necessary to be slowly increasing, but this is not true
for complex (or real, signed) measures.
\par
By  the ``support'' of a pure point measure $\mu$ we mean the  
countable set of the non-zero atoms of $\mu$. This should not be confused with the
notion of support in the sense of distributions, which is always a closed set.
\par
The Fourier transform on $\R^n$ will be normalized as follows:
$$\ft \varphi (t)=\int_{\R^n} \varphi (x) \, e^{-2\pi i\langle t,x\rangle} dx.$$
\par
We denote by $\supp(\varphi)$ the closed support of a Schwartz function $\varphi$,
and by $\spec (\varphi)$ the closed support of its Fourier transform $\ft\varphi$.
\par
If $\alpha$ is a temperate distribution, and $\varphi$ is a Schwartz function on $\R^n$,
then $\dotprod{\alpha}{\varphi}$ will denote the action of $\alpha$ on ${\varphi}$.
The Fourier transform $\ft\alpha$ of the distribution $\alpha$ is defined by 
$\dotprod{\ft\alpha}{\varphi} = \dotprod{\alpha}{\ft\varphi}$.
\par
A distribution $\alpha$ is called \emph{positive} if  $\dotprod{\alpha}{\varphi} \geq 0$
for any Schwartz function $\varphi \geq 0$. It is well-known that if $\alpha$ is a positive distribution,
then it is a positive measure. A distribution $\alpha$ is called \emph{positive-definite} if $\ft\alpha$ is a positive 
distribution. 
\par
For a set $A \subset \R^n$ we denote by $\# A$ the number of elements in $A$, and
by $\mes(A)$  or $|A|$ the Lebesgue measure of $A$.

\subsection{Measures}
We  will need some basic facts about measures in $\R^n$.
\begin{lem}
	\label{AL.4}
	Let $\mu$ be a measure on $\R^n$,  whose support $\Lambda$ is a uniformly
	discrete set. Assume that $\ft{\mu}$ is a slowly increasing measure. Then 
	\begin{equation}
	 \sup_{\lambda\in\Lambda} |\mu(\lambda)|<\infty,
		\label{AL.4.1}
	\end{equation}
	and so $\mu$ is a translation-bounded measure. 
\end{lem}
This can be proved in a similar way to \cite[Lemma 2]{lo2}.
\begin{lem}
	\label{AL.10}
Let $\mu$ be a measure on $\R^n$, whose support $\Lambda$ is a uniformly discrete
set. Assume that $\hat{\mu}$ is a slowly increasing measure, with at least one non-zero
atom. Then $\Lambda$ is a relatively dense set in $\R^n$. 
\end{lem}
\begin{proof}
	By Lemma \ref{AL.4}, the measure $\mu$ is translation-bounded. Let us suppose that
	$\Lambda$ is not relatively dense, and show that this implies that $\hat{\mu}(\{s\})=0$
for every $s \in \R^n$.
\par
	Choose a Schwartz function $\varphi$ whose Fourier transform $\hat{\varphi}$ has compact
	support, and $\hat{\varphi}(0)=1$. For each $0<\delta<1$ define
		$\varphi_\delta(t):= \delta^n\varphi(\delta t)$. Then we have 
	\begin{equation}
		\hat{\mu}(\{s\})=\lim_{\delta\to 0}\int \hat{\varphi}_\delta(t-s)e^{2\pi
		i\dotprod{x}{t-s} }d\hat\mu(t)
		\label{AL.10.1}
	\end{equation}
	uniformly with respect to $x\in\R^n$. On the other hand, 
	\begin{equation}
		\int \hat{\varphi}_\delta(t-s)e^{2\pi i\dotprod{x}{t-s}}d\hat{\mu}(t) = 
		\int \varphi_\delta(x-y)e^{-2\pi i\dotprod{s}{y}}d\mu(y).
		\label{AL.10.2}
	\end{equation}
	If $\Lambda$ is not relatively dense, then for any $R>0$ there is $x\in \R^n$ such
	that the ball $x+B_{R}$ does not intersect $\Lambda$. Using the
	translation-boundedness of $\mu$ this implies that for any $\delta>0$, there are
	values of $x$ for which the right-hand side of \eqref{AL.10.2} is arbitrarily
	close to zero. Hence the limit in \eqref{AL.10.1} must be zero, which proves the
	claim.
\end{proof}

\subsection{Interpolation}
For a compact set $\Omega \subset \R^n$, we denote by $\B(\Omega)$ the   Bernstein space
consisting of all bounded, continuous functions $f$ on $\R^n$ such that the distribution $\ft f$
is supported by $\Omega$. 
A set $\Lam\subset\R^n$ is called an \emph{interpolation set} for the space $\B(\Omega)$ if for every bounded sequence
 $\{c_\lam\}$, $\lam\in\Lam$, there exists at least one $f \in \B(\Omega)$ satisfying
$f(\lam)=c_\lam$ $(\lam\in\Lam)$. It is well-known that such
$\Lam$ must be a uniformly discrete set.
\par
The following result is due to Ingham for $n=1$, and  Kahane for $n>1$, see
\cite{ou3}. 
\begin{thm}
	There is a constant $C$ which depends on the dimension $n$ only, such that if
	$\Lambda$ is a uniformly discrete set in $\R^n$, $d(\Lambda)\ge a> 0$, then
	$\Lambda$ is an interpolation set for $\mathcal{B}(\Omega)$ where $\Omega$ is any closed
	ball of radius $C/a$. 
	\label{AL.11}
\end{thm}
As a consequence we obtain:
\begin{corollary}
	\label{AC.12}
	There is a constant $C$ which depends on the dimension $n$ only, such that if a measure
	$\mu$ is supported on a uniformly discrete set $\Lambda\subset \R^n$,
	$d(\Lambda)\ge a> 0$, and if the distribution $\hat{\mu}$ vanishes on a ball of radius $C/a$, then $\mu=0$. 
\end{corollary}

\begin{proof}
Suppose that $\hat{\mu}$ vanishes on the ball $B_R$, where $R:=(C+1)/a$ and
$C$ is the constant from \thmref{AL.11}
 (we may assume that the ball is centered at the origin). 
Let $\Omega=\{x : |x| \leq C/a\}$. 
Given $\lambda\in\Lambda$, there is $f\in \mathcal{B}(\Omega)$  such that $f(\lambda)=1$ and $f$ vanishes on
	$\Lambda\setminus \{\lambda\}$. Define $\varphi(x):= f(x)\psi(x)$, where $\psi$ is a
	Schwartz function such that $\psi(\lambda)=1$ and $\spec(\psi)\subset B_{1/a}$.
	Then $\varphi$ is a Schwartz function satisfying
\[
\varphi(\lambda)=1, \quad
	\text{$\varphi(\lambda')=0$ for all $\lambda' \in \Lam \setminus \{\lam\}$},  \quad
\spec(\varphi)\subset B_R.
\]
Hence
\[ \mu(\lambda)=\int
	\overline{\varphi(x)}d\mu(x)=\dotprod{\hat{\mu}}{\overline{\hat{\varphi}}}=0.
\]
	This holds for any $\lambda\in\Lambda$, so we obtain $\mu=0$. 
\end{proof}

% ====================================================================  

\section{Auxiliary measures $\nu_h$}
Let $\mu$ be a measure on $\R^n$ satisfying \eqref{RM.1} and \eqref{RM.2},
and assume that the support $\Lambda$ is a uniformly discrete set.
By Lemma \ref{AL.4}, the atoms of $\mu$ are bounded, so $\mu$ is a translation-bounded measure. 

 For each $h\in S-S$ we denote 
\begin{equation}
	S_h:= S\cap (S-h)= \left\{ s\in S\,:\, s+h\in S \right\}, 
	\label{AL.3.1}
\end{equation}
which is a non-empty subset of $S$, and we introduce a new measure 
\begin{equation}
	\nu_{h}:= \sum_{s\in S_h}\hat{\mu}(s)\,\overline{\hat{\mu}(s+h)}\,\delta_s.
	\label{AL.3.2}
\end{equation}
Notice that it is a non-zero, slowly increasing measure, whose support is the set $S_h$.
\begin{lem}
	\label{AL.6}
	The Fourier transform $\hat{\nu}_h$ of the measure $\nu_{h}$ is also a measure,
	which is translation-bounded and supported by the closure of the set
	$\Lambda-\Lambda$.
\end{lem}

This is an elaborated version of \cite[Lemma 12]{lo2}. That lemma stated that $\spec(\nu_h)$
does not intersect the punctured ball $B_a\setminus\left\{ 0 \right\}$, where $a:=d(\Lambda)>0$.
However, the result there was formulated with the roles of
	$\Lambda$ and $S$ interchanged, and under the stronger assumption that $\Lambda$,
	$S$ are both uniformly discrete sets. 

	\begin{proof}[Proof of Lemma \ref{AL.6}] 
		Fix a Schwartz function $\varphi$, whose Fourier transform $\hat{\varphi}$ has
		compact support and $\hat{\varphi}(0)=1$. For each $0<\delta<1$ we denote
		$\varphi_\delta(x):= \delta^n\varphi(\delta x)$, and define the measure
		\begin{equation}
			\nu_h^{(\delta)} (t) := \overline{\left( \hat{\mu} \ast \hat{\varphi}_\delta
			\right)(t+h)}\cdot
			\hat{\mu}(t).
			\label{AL.6.1}
		\end{equation}
It is a slowly increasing measure, supported by $S$, which tends to $\nu_h$ in the space of temperate
distributions as $\delta\to 0$. Hence $\hat{\nu}_h^{(\delta)}$ tends to $\hat{\nu}_h$ in the
same sense as $\delta\to 0$. 
\par
We will show that $\hat{\nu}_h^{(\delta)}$ is a
translation-bounded measure, and moreover 
\[ \sup_{x\in\R^n} |\hat{\nu}_{h}^{(\delta)}| (x+B_1) \] 
is bounded by some constant $C(\mu,\varphi)$ which depends on $\mu$ and $\varphi$ only
(in particular, it does not depend on $\delta$).
 Indeed, the Fourier transform of $\nu_h^{(\delta)}$ is the
measure 
\begin{equation}
	\hat{\nu}_h^{(\delta)}= \overline{(\varphi_\delta\cdot e_{-h}\cdot \mu)(x)} \ast
	\mu(-x),
	\label{AL.6.2}
\end{equation}
where $e_{-h}(x):= e^{-2\pi i \dotprod{h}{x}}$. Hence, we have
\[ \sup_{x\in \R^n} |\hat{\nu}_h^{(\delta)}| (x+B_1) \le 
\Big\{ \sup_{x\in \R^n} |\mu| (x+B_1) \Big\} \Big\{
\int |\varphi_\delta (x)| \, |d\mu(x)| \Big\} \le C(\mu,\varphi), \]
since $\mu$ is  translation-bounded. Letting $\delta\to 0$ this implies
that the limit $\hat{\nu}_h$ is also a translation-bounded measure,
and in fact $|\hat{\nu}_h| (x+B_1) \le C(\mu,\varphi)$ for all $x\in \R^n$.
\par
Finally, it follows from \eqref{AL.6.2} that the measure $\hat{\nu}_h^{(\delta)}$ is supported by
 $\Lambda-\Lambda$. Hence its limit as $\delta\to 0$ must be supported by the closure
 of $\Lambda-\Lambda$, which ends the proof. 
	\end{proof}

% ====================================================================  

\section{Positive-definite measures}
\label{sec:pos-def}
\subsection{ } \label{5.1a}
In this section we consider positive-definite measures whose supports are
uniformly discrete sets. We establish a dichotomy concerning the
discreteness of the spectrum: either it is also  uniformly discrete, or it is
``non-discrete'' in a strong sense. 
\begin{thm}
	\label{AT.5}
	Let $\mu$ be a positive-definite measure on $\R^n$ satisfying \eqref{RM.1} and
	\eqref{RM.2}, and assume that the support $\Lambda$ is a uniformly discrete set. Then, either 
\begin{enumerate-math}
\item
$S$ is also  uniformly discrete; or 
\item
 $S$ has a relatively dense  set of accumulation points.
\end{enumerate-math}
\end{thm}

In particular, it follows that if the spectrum $S$ is a discrete closed set, then it must
be uniformly discrete. Hence \thmref{RT.0} applies to this situation, and yields that the
measure $\mu$ is representable in the form \eqref{RM.3}. So \thmref{RT.5} follows.

\begin{remark}
Actually we will prove that there is a constant
	$C$ which depends on the dimension $n$ only, such that if the spectrum
$S$ is not uniformly discrete, then every ball of radius
	$C/d(\Lam)$ contains infinitely many points of $S$.
\end{remark}

\subsection{}
We will need the following auxiliary lemma:
\begin{lem}
	\label{AL.1}
	There is a real-valued Schwartz function $\varphi$ on $\R^n$ which has  the following
	properties: 
	\begin{enumerate-math}
	\item \label{it_al.1.1} There is $R$ such that $\varphi(x)>0$ for $|x|\ge R$;
	\item \label{it_al.1.2} $\int \varphi(x)dx=0$;
	\item \label{it_al.1.3} $ \spec(\varphi)$ is contained in $B_1$ (the open unit ball).
	\end{enumerate-math}
\end{lem}
\begin{proof}
	Choose a Schwartz function $\psi > 0$ whose Fourier transform $\hat{\psi}$ is
	supported in the ball $\left\{t:  |t|\le 1/3 \right\}$. Define $\varphi:=\alpha \psi
	-\beta \psi^2$, where $\alpha:= \left( \int \psi \right)^{-1}$ and $\beta:=\left(
	\int \psi^2 \right)^{-1}$. It is easy to verify that all the properties
	\ref{it_al.1.1}, \ref{it_al.1.2} and \ref{it_al.1.3} are satisfied.
\end{proof} 
\subsection{}
\begin{proof}[Proof of \thmref{AT.5}]
	Assume that $S$ is not uniformly discrete.  Let $\varphi$ be the function given by
	\lemref{AL.1}, and $R$ be the number from property \ref{it_al.1.1} of that lemma.
	We will show that 
	any closed ball of radius $C/a$ contains infinitely many points of $S$, where
$C : = R+1$ and $a:=d(\Lambda)>0$ (notice that the constant $C$ indeed depends on the dimension $n$ only). 
In particular this will show that the set of accumulation points of $S$ must be relatively dense.

\par
	By multiplying $\mu$ by an exponential (which corresponds to translation of
	$\hat{\mu}$), it will be enough to show that the ball $\{|t| \leq C/a\}$
contains  infinitely many points of $S$. So, suppose to the
	contrary that this does not hold, namely this ball
	contains only finitely many points of $S$. 
\par
 Since $S$ is not uniformly discrete,
	the set $S-S$ contains elements $h\neq 0$ arbitrarily close to zero. Hence we may choose
	$h\in S-S$ such that the set $S_{h}$ defined by \eqref{AL.3.1} does not intersect
	the ball $\left\{ |t| < R/a \right\}$. 	It follows that the measure $\nu_h$
	in \eqref{AL.3.2} is a non-zero positive measure, whose support is
	contained in $\left\{ |t|\ge R/a \right\}$. Using property \ref{it_al.1.1} from
	\lemref{AL.1} this implies that 
	\begin{equation}
		\label{AT.5.1}
		\int \varphi(ax)\,d\nu_h(x) > 0.
	\end{equation}
	On the other hand, we have 
	\begin{equation}
		\label{AT.5.2}
		\int
		\varphi(ax)\,d\nu_h(x)=a^{-n}\int \hat{\varphi}(-t/a) \, d\hat{\nu}_h(t).
	\end{equation}
	By \lemref{AL.6}, $\hat{\nu}_h$ is a measure, supported by the closure of
$\Lam-\Lam$. Since $\Lam$ is uniformly discrete, 
 this closure is contained in the set $\{0\}\cup
	\{|t|\ge a\}$.
 But from  properties \ref{it_al.1.2} and \ref{it_al.1.3} in
	\lemref{AL.1} it follows that the function $\hat{\varphi}(-t/a)$ vanishes on this set.
	Hence, the right-hand side of \eqref{AT.5.2} must vanish, 
	in contradiction with \eqref{AT.5.1}.  
\end{proof}

% ====================================================================  

\section{Spectra with finite density}
In this section we prove \thmref{RT.6}. We will show that under the conditions in the
theorem, the spectrum $S$ of the measure $\mu$ must be a uniformly discrete set. Then the final conclusion that
$\mu$ is of the form \eqref{RM.3} can be deduced from \thmref{RT.0}. 
\subsection{}
For a set $\Lambda\subset \R $ we denote 
\[ \rho(\Lambda):= \sup_{x\in\R}\# (\Lambda\cap [x,x+1]). \] 
Notice that $\rho(\Lambda)<\infty$ if and only if $\Lambda$ is a finite union of uniformly
discrete sets. 
\par
We will need the following notion of ``lower density'' of a set $\Lambda \subset \R$, defined by 
\[ D_{\#}(\Lambda):=\liminf_{R\to\infty}\frac{\#(\Lambda\cap (-R,R))}{2R}. \]
Clearly we have $D_{\#}(\Lambda)\le\rho(\Lambda)$. It will be useful below to extend the
definition of the density $D_{\#}$ also to multi-sets $\Lambda \subset \R$, that is, to the
case when points in $\Lambda$ occur with multiplicities. Notice that $D_{\#}$ is
super-additive in the sense that 
\[ D_{\#}(A\cup B)\ge D_{\#}(A)+D_{\#}(B), \]
where the union $A\cup B$ is understood in the sense of multi-sets.
\subsection{ }
The following result is a more general version of \cite[Proposition 4]{lo2}.
\begin{prop}
	\label{BP.1}
	Let $\Lambda\subset \R$ be a set with $\rho(\Lambda)\le M< \infty$. Assume that
	$\Lambda$ supports a non-zero, slowly increasing  measure $\mu$, such that the distribution
	$\hat{\mu}$ vanishes on the open interval $(0,a)$ for some $a> 0$. Then 
	\[ D_{\#}(\Lambda)\ge c(a,M), \]
	where $c(a,M)>0$ is a constant which depends on $a$ and $M$ only. 
\end{prop}
This can be deduced from the following lemma:
\begin{lem}
	\label{BL.2}
	Let $\Lambda$ be a finite subset of $(-R,R)\setminus (-1,1)$, such that
	$\rho(\Lambda)\le M$, and let $a>0$. There is $c(a,M)>0$ such that if
	$(\#\Lambda)/(2R)<c(a,M)$, then one can find a Schwartz function $\varphi$ with
	the following properties: 
	\[ \varphi(0)=1, \;\; \varphi(\lambda)=0 \;\; (\lambda\in \Lambda), \;\; \spec(\varphi)\subset
	(0,a), \;\; \sup_{|x|\ge R} |\varphi(x)|\le 1. \]
\end{lem}
The proof of \lemref{BL.2}, as well as the deduction of \propref{BP.1} from this lemma,
can be done in a way similar to \cite[Section 4.1]{lo2}, and we omit the details.
\subsection{} 
\begin{lem}
	\label{BL.3}
	Let $S\subset \R$ be a set with $\rho(S)<\infty$. Suppose that there is
	$c=c(S)>0$ such that $D_{\#}(S_h)>c$ for every $h\in S-S$. Then
	$\rho(S-S)<\infty$. 
\end{lem}
\begin{proof}
	Let $x\in\R$, and suppose that $h_1,\dots, h_N$ are distinct points in the set
	$(S-S)\cap [x,x+1]$. Since the lower density $D_{\#}$ is super-additive, we have 
	\[ cN \le \sum_{j=1}^{N}D_{\#}(S_{h_j})\le D_{\#}\Big(
	\bigcup_{j=1}^{N}S_{h_j} \Big), \]
where the union is understood in the sense of multi-sets. Notice that each point in this
union occurs with multiplicity not greater than $\rho(S)$. It follows that $cN\le
\rho(S)  D_{\#}(S)$, which shows that the set $S-S$ cannot have more than
$\rho(S)  D_{\#}(S)/c$ elements in any closed interval of length 1. Hence $\rho(S-S)<\infty$,
which proves the claim.
\end{proof}
\subsection{}
\begin{proof}[Proof of \thmref{RT.6}] 
We assume that $\mu$ is a measure on $\R$ satisfying \eqref{RM.1} and
\eqref{RM.2}, where $\Lambda$ is a uniformly discrete set, and $S$ is a set with
$\rho(S)<\infty$.
\par
 For each $h\in S-S$, let $\nu_{h}$ be the measure
defined by \eqref{AL.3.2}. By \lemref{AL.6} the Fourier transform $\hat{\nu}_h$ is
supported by the closure of the set $\Lambda-\Lambda$. Since $\Lambda$ is uniformly
discrete this implies that $\hat{\nu}_{h}$ vanishes on the open interval $(0,a)$, where
$a:=d(\Lambda)>0.$ As the measure $\nu_h$ is supported by $S_h$, it follows from \propref{BP.1} that 
$ D_{\#}(S_h)\ge c$, where $c>0$ is a constant which depends on $d(\Lambda)$ and $\rho(S)$. 
\par
Since this holds  for every $h\in S-S$, \lemref{BL.3} allows us to deduce that $\rho(S-S)<\infty$. In particular, the set
$S-S$ has no accumulation point at zero, so there is $\delta>0$ such that $(S-S)\cap
(-\delta,\delta)=\left\{ 0 \right\}$. Hence $S$ must be uniformly discrete, and
in fact $d(S)\ge \delta$. 
\par
Once we have concluded that 
$\Lambda$ and $S$ are both uniformly discrete sets, we can apply \thmref{RT.0} which yields that the
measure $\mu$ is representable in the form \eqref{RM.3}.
\end{proof}

% ====================================================================  

\section{Meyer sets}
\subsection{}
In this section we show that if the support $\Lam$  satisfies a  stronger  
              discreteness condition than in \thmref{AT.5},  then the conclusion of this  theorem
              remains true without  any additional  positivity restriction.
\begin{definition*}
	A set $\Lambda\subset \R^n$ is called a \emph{Delone set} if $\Lambda$ is both a
	uniformly discrete and relatively dense set. 
\end{definition*}
\lemref{AL.10} implies that a uniformly discrete set $\Lam$ which supports a measure $\mu$,
whose Fourier transform $\ft\mu$ is  a pure point
measure, must be a Delone set.
\begin{definition*}
	A set $\Lambda\subset \R^n$ is called a \emph{Meyer set} if the following two conditions
	are satisfied: 
	\begin{enumerate-math}
	\item $\Lambda$ is a Delone set;
	\item There is a finite set $F$ such that $\Lambda-\Lambda\subset \Lambda+F$. 
	\end{enumerate-math}
\end{definition*}
The concept of Meyer set was introduced in \cite{mey0, mey1} in connection with problems
in harmonic analysis. After the experimental discovery of quasicrystalline materials in
the middle of 80's, Meyer sets have been extensively studied as mathematical models of
quasicrystals. 
\par
There are some equivalent forms of the definition of a Meyer set,
see  \cite{moo}. In particular, the following is true (Lagarias \cite{lag1}):
\par
\emph{A Delone set $\Lam$ is a Meyer set if and only if
$\Lambda-\Lambda$ is  uniformly discrete}
\par
\noindent
           (a simplified version of the proof of this equivalence can be found in \cite[Lemma 8]{lo2}).

\subsection{}
Now we show that if a measure $\mu$ is supported by a Meyer set $\Lam$, then the
dichotomy phenomenon for the spectrum $S$ is valid: either $S$ is uniformly discrete, or
it is non-discrete with a relatively dense set of accumulation points.  
\begin{thm}
	\label{CT.1}
	Let $\mu$ be a measure on $\R^n$ satisfying \eqref{RM.1} and \eqref{RM.2},
	and assume that the support $\Lam$ is a Meyer set. Then the same
	conclusion as in \thmref{AT.5} holds.
\end{thm}

\begin{proof}
As in the proof of \thmref{AT.5}, it will
	be enough to prove that for every $h\in S-S$, the set $S_h$ must intersect any ball
	of radius $C/a$, but now we will take $C$ to be the number from \corref{AC.12}, and
	$a:=d(\Lambda-\Lambda)>0$.
\par
And indeed, by \lemref{AL.6} the measure $\nu_h$
	defined by \eqref{AL.3.2} is a non-zero measure, whose Fourier transform
	$\hat{\nu}_h$ is a translation-bounded measure supported by the set
	$\Lambda-\Lambda$ (this set is uniformly discrete, so it is not necessary to consider
its closure). Now \corref{AC.12}, applied to the measure
	$\hat{\nu}_h$,  implies that $\nu_h$ cannot vanish on a ball of radius
	$C/a$. Hence  $S_h$ must intersect any such a ball, which completes the proof.
\end{proof}

\begin{remark}
The proof in fact shows that there is a constant
	$C$ which depends on the dimension $n$ only, such that if the spectrum
$S$ is not uniformly discrete, then every ball of radius
	$C/d(\Lam-\Lam)$ contains infinitely many points of $S$.
\end{remark}

\subsection{}
Next, we deduce \thmref{RT.7} from the above result. Suppose that $\Lambda-\Lambda$ is a
uniformly discrete set, and that $S$ is  a discrete closed set. Hence $S$ 
has no finite accumulation points, so it follows from \thmref{CT.1} that $S$
  must be uniformly discrete. 
\par
However, to complete the conclusion of \thmref{RT.7} it still remains to show that  $\mu$ 
is representable in the form \eqref{RM.3}. Here one cannot directly
apply \thmref{RT.0}, since the measure was assumed to be neither positive nor positive-definite. 
In order to obtain \eqref{RM.3} we use instead the following version of
\thmref{RT.0}, proved in \cite{lo1}:
\begin{thm}
	\label{RT.9}
	Let $\mu$ be a measure on $\R^n$ satisfying \eqref{RM.1} and \eqref{RM.2}. 
	Assume that the sets $\Lambda-\Lambda$ and $S$ are both uniformly discrete. Then the
	conclusion of \thmref{RT.0} holds.
\end{thm}
 Combining Theorems \ref{CT.1} and \ref{RT.9} thus implies the full assertion of \thmref{RT.7}.

% ====================================================================  

\section{Nowhere dense spectra}

\label{subsec:notdense}
In this section we prove \thmref{RT.15}. We will construct a non-periodic
measure $\mu$ supported on  a Meyer set $\Lam\subset \R^n$,
such that the spectrum $S$ is not dense in any ball. 
Moreover, the measure $\mu$ in the construction 
 is positive-definite, and both $\mu$ and $\ft\mu$
are translation-bounded measures.

\subsection{}
 Let $\Gamma$ be a lattice in $\R^n\times \R^m$, and let
	$p_1$ and $p_2$ denote the projections onto $\R^n$ and $\R^m$ respectively. Assume
	that the restrictions of $p_1$ and $p_2$ to $\Gamma$ are injective, and that their
	images are dense in $\R^n$ and $\R^m$ respectively.
	Let $\Gamma^{*}$ be the dual
	lattice, then the restrictions of $p_1$ and $p_2$ to $\Gamma^{*}$ are also
	injective and have dense images.
\par
 If $\Omega$ is a bounded open set in
	$\R^m$, then the set 
	\begin{equation}
\Lambda(\Gamma,\Omega):= \left\{ p_1(\gamma)\,:\,
	\gamma\in \Gamma, \, p_2(\gamma)\in \Omega \right\} 
		\label{RP.9.0}
	\end{equation}
is called the ``model set'',
	or the ``cut-and-project set'', associated to the lattice $\Gamma$ and to the
	``window'' $\Omega$. 
It is well-known that any model set is a Meyer set.
\par
Meyer observed \cite[p.\ 30]{mey0} (see also \cite{mey2}) that
	model sets provide examples of non-periodic uniformly discrete sets, which support
	a measure $\mu$ such that the Fourier transform $\ft\mu$ is also a pure point measure. Such a
	measure may be obtained by choosing a Schwartz function $\varphi$ on
$\R^m$ such that $\supp(\hat{\varphi})\subset \Omega$, and taking 
	\begin{equation}
		\mu = \sum_{\gamma\in\Gamma} \hat{\varphi} ( p_2(\gamma)  )
		\delta_{p_1(\gamma)}.
		\label{RP.9.1}
	\end{equation}
	It is not difficult to verify that this is a translation-bounded measure, whose Fourier transform is the
	(also translation-bounded) pure point measure
	\begin{equation}
		\hat{\mu}=\frac{1}{\det \Gamma}
		\sum_{\gamma^{*}\in\Gamma^{*}}\varphi( p_2(\gamma^{*})
		)\delta_{p_1(\gamma^{*})}.
		\label{RP.9.2}
	\end{equation}
\par
However, the compact support of $\hat{\varphi}$ implies that $\varphi$ is an entire function,
and so $\varphi$ cannot also be supported on a bounded set. Hence
 the spectrum of the measure $\mu$ is only known to be contained in $p_1(\Gam^*)$, 
and so it is generally everywhere dense in $\R^n$.
\par
Nevertheless,  we will see that one can construct a function $\varphi$ with sufficiently many zeros,
in such a way that the non-zero atoms in \eqref{RP.9.2} in fact lie in a nowhere dense set.

\subsection{} \label{section_8.2}
A set $\Lambda\subset\R^n$ is said to have a \emph{uniform density} $D(\Lambda)$ if 
\[ \frac{ \#\left( \Lambda \cap (x+B_R) \right)}{|B_R|} \to D(\Lambda) \]
as $R\to \infty$ uniformly with respect to $x\in \R^n$. 
\begin{lem}
	\label{EL.2}
	Let $\Lambda=\Lambda(\Gamma,\Omega)$ be a model set such that the boundary of
	$\Omega$ is a set of Lebesgue measure zero in $\R^m$. Then $\Lambda$ has uniform
	density 
	\[ D(\Lambda) = \frac{\mes(\Omega)}{\det(\Gamma)}. \]
\end{lem}
A proof of this fact can be found e.g.\ in \cite[Proposition 5.1]{matei-meyer-simple}. 
\subsection{} \label{8 .3}
We will  now assume that $m=1$, that is, $\Gam$ is a lattice in $\R^n \times \R$.
The following theorem is the main result of this section.
\begin{thm}
	\label{ET.1}
	For any $\varepsilon > 0$ there is a non-zero Schwartz function $\varphi\ge 0$ on
	$\R$, such that: 
	\begin{enumerate-math}
	\item The measure $\mu$ in \eqref{RP.9.1} is supported by the model set
		$\Lambda(\Gamma,(-\varepsilon,\varepsilon))$;
	\item The spectrum of $\mu$ is a nowhere dense set in $\R^n$. 
	\end{enumerate-math}
\end{thm}
Observe that the support of $\mu$ cannot be covered by a finite 
union of translates of any lattice, since it contains a model set. Hence this 
result implies \thmref{RT.15}. The condition $\varphi\ge 0$ guarantees  that $\mu$ is a
positive-definite measure. 
\par
The proof of \thmref{ET.1} depends on the following: 
\begin{lem}
	\label{EL.3}
	For each $j\ge 1$ let $Q_j\subset\R$ be a   set  with
	uniform density $D(Q_j)$. Assume that 
	\begin{equation}
		\label{EL.3.2}
		\sum_{j\geq 1} D(Q_j)< a.
	\end{equation}
	Then one can find positive numbers $T_j$ 
and a non-zero Schwartz function $\varphi$ on $\R$ such that: 
	\begin{enumerate-math}
	\item $\spec(\varphi)\subset (0,a)$;
	\item $\varphi$ vanishes on the set $Q$ defined by
		\begin{equation}
			\label{EL.3.1}
			Q:= \bigcup_{j\geq 1} Q_j'  , \qquad
			Q_j':=Q_j\setminus(-T_j,T_j).
		\end{equation} 
	\end{enumerate-math}
\end{lem}
Before we prove \lemref{EL.3}, let us first show how to deduce \thmref{ET.1} from it.
\begin{proof}[Proof of \thmref{ET.1}]
	Let $\{x_j\}$ $(j\ge 1)$ be a sequence of points which are dense in $\R^n$. For
	each $j$, choose an open ball $B_j$ centered at the point $x_j$, in such a way
	that 
	\begin{equation}
		\label{EL.4.1}
		\sum_{j\ge 1} \mes(B_j)< \frac{\varepsilon}{\det(\Gamma)}.
	\end{equation}
	Consider the sets $Q_j \subset \R$ defined by 
	\[ Q_j := \left\{ p_2(\gamma^{*})\; :\; \gamma^{*}\in\Gamma^{*}, \;
		p_1(\gamma^{*})\in B_j  \right\}.
	   \]
	Then each $Q_j$ is a model set, with uniform density 
	\[ D(Q_j)= \det(\Gamma) \mes(B_j) \]
	according to \lemref{EL.2}. Due to \eqref{EL.4.1} this implies that 
	\[ \sum_{j\ge 1} D(Q_j)< \varepsilon.  \]
	\lemref{EL.3} therefore gives a sequence $\{T_j\}$ of positive numbers,
	and  a non-zero Schwartz function $\psi$ on $\R$, 
	$\spec(\psi) \subset (0,\varepsilon)$,  such that $\psi$  
	vanishes on the set $Q$ in \eqref{EL.3.1}.
	Hence  $\varphi := |\psi|^2 \geq 0$  is also a Schwartz function vanishing on
	$Q$, and $\spec(\varphi)\subset (-\varepsilon,\varepsilon)$.
\par
	For each $j\ge 1$ there are only finitely many points 
	of the lattice $\Gamma^{*}$ lying in the set $B_j\times (-T_j,T_j)$, so we
may choose an open ball $\Omega_j$ contained in
	$B_j$ such that $\Omega_j\times (-T_j,T_j)$ has no points in common
	with $\Gamma^{*}$. Notice that the set 
	\[ \Omega := \bigcup_{j\ge 1} \Omega_j \]
	is an open, dense set in $\R^n$.
\par
	We claim that the spectrum of the measure
	\eqref{RP.9.1} does not intersect the set $\Omega$. Indeed, by \eqref{RP.9.2}, an
	element of the spectrum is a point of the form $p_1(\gamma^{*})$, where
	$\gamma^{*}\in \Gamma^{*}$ and $\varphi ( p_2(\gamma^{*})  ) \ne 0$. 
	If $p_1(\gamma^{*})\in \Omega_j$ for some $j$, then we must have
	$|p_2(\gamma^{*})|\ge T_j$. Hence 
\[ p_2(\gamma^{*})\in Q_j\setminus(-T_j,T_j) \subset Q,
\]
 which is not
	possible as $\varphi$ vanishes on $Q$. 
\par
	We conclude that the spectrum $S$ of the
	measure $\mu$ is contained in the closed, nowhere dense set $\R^n\setminus
	\Omega$. On the other hand, the support $\Lambda$ of the
	measure  is contained in the model set 
	$\Lambda (\Gamma, (-\varepsilon,\varepsilon) )$, so this completes the proof. 
\end{proof}
\subsection{ } \label{section_8.4}
It remains to prove \lemref{EL.3}. For this we will use
the celebrated Beurling and Malliavin theorem, see \cite{bm67}.
\par
First we recall the definition of the \emph{Beurling-Malliavin upper
density} (there are several equivalent ways to define this
density). By a \define{substantial system} of intervals  we mean a
system $\{I_k\}$ of disjoint open intervals on $\R$, such that
$\inf_{k}|I_k|>0$, and 
\[	\sum_{k}\left( \frac{|I_k|}{1+\dist(0,I_k)} 
	\right)^2 = \infty. \]
\par
If $\Lambda\subset\R$ is a discrete closed set, then its
Beurling-Malliavin upper density $D^{*}(\Lambda)$ is defined to be the
supremum of the numbers $d>0$, for which there exists a
substantial system $\{I_k\}$ satisfying 
\[	\frac{\# (\Lambda \cap I_k)}{|I_k|}\ge d \]
for all $k$. If for any $d>0$ no such a system  $\{I_k\}$ exists, then $D^{*}(\Lambda)=0$.
\begin{thm}[Beurling and Malliavin \cite{bm67}]
\label{ET.2}
Let $\Lambda\subset \R$ be a discrete closed set. Then for any
$a>D^{*}(\Lambda)$ one can find a non-zero function $\varphi\in
L^{2}(\R)$ such that: 
\begin{enumerate-math}
\item $\spec(\varphi)\subset (0,a)$;
\item $\varphi$ vanishes on $\Lambda$. 
\end{enumerate-math}
\end{thm}
By multiplying $\varphi$ by a Schwartz function with a sufficiently small 
spectrum, it is clear that one may assume the function $\varphi$ in this 
theorem to belong to the Schwartz class.
\par
\begin{remark}
It was also proved by 
Beurling and Malliavin that if $a<D^{*}(\Lambda)$ then no such
$\varphi$ exists; however we will not use this part of their
result. 
\end{remark}
\begin{proof}[Proof of \lemref{EL.3}]
Choose a sequence $\{\gamma_j\}$ and a number $d$ such that 
\[ D(Q_j) < \gamma_j  \quad \text{and} \quad \sum_{j\ge 1} \gamma_j < d <a. \]
For each $j$  find a  number $M_j$, such
that for any interval $I $ of length $|I|\ge M_j$ we have 
\begin{equation}
	\label{EL.3.7}
	\#(Q_j\cap I) \leq \gamma_j \, |I|.
\end{equation}
Define
\[ T_j := M_j^3 +M_j. \]
We  claim that if $Q$ is the set given by \eqref{EL.3.1},
then $D^{*}(Q)\le d$.
\par
Let $\{I_k\}$ be a substantial
system of intervals on $\R$. Observe first that there must exist at least one
interval $I_k$ satisfying
\begin{equation}
	\label{EL.3.8}
	 \dist(0, I_k) \le |I_k|^3.
\end{equation}
For otherwise, using the assumption that $\inf_k
|I_k|>0$, this would imply that 
\[ \sum_k \left( \frac{|I_k|}{1+\dist(0,I_k)} \right)^2 \le 
\sum_k \big(1+\dist (0,I_k) \big)^{-4/3} < \infty,\]
which is not possible since the system $\{I_k\}$ is substantial.
\par
Now consider an interval $I_k$ from the system, satisfying \eqref{EL.3.8}. We will show
that 
\begin{equation}
	\label{EL.3.9}
	\frac{\# (Q\cap I_k)}{|I_k|}<d.
\end{equation}
Indeed, we have
\begin{equation}
	\label{EL.3.10}
	\#(Q\cap I_k) \le \sum_{j\ge 1}\# (Q_j'\cap I_k).
\end{equation}
Notice that for each $j$ such that $|I_k|<M_j$, it follows from \eqref{EL.3.8} that 
\[I_k\subset (-M_j^3-M_j,M_j^3+M_j) = (-T_j,T_j),\]
and hence $I_k$ contains no points in common with $Q_j'$.
On the other hand, for each $j$  such that $|I_k|\ge M_j$ we have 
\[\#(Q_j'\cap I_k) \leq \gamma_j \, |I_k|\]
 according to \eqref{EL.3.7}. Hence 
\begin{equation}
	\label{EL.3.10.1}
 \sum _{j\ge 1}\# (Q_j' \cap I_k) \leq   \sum_{j \, :\,  M_j \leq |I_k|} \gamma_j \, |I_k|  < d \, |I_k|.
\end{equation}
Combining \eqref{EL.3.10} and \eqref{EL.3.10.1} confirms that \eqref{EL.3.9} holds.
\par
We have thus shown that any substantial system $\{I_k\}$ contains
intervals for which \eqref{EL.3.9} holds. Hence
$D^{*}(Q)\le d< a$. The proof is now concluded by \thmref{ET.2}.
\end{proof}

% ====================================================================  

\section{Hof's quasicrystals}
There exist also other approaches to the concept of
quasicrystals. One of them, which is due to Hof \cite{hof}, was
studied by many authors, see for example
\cite{lag2}, \cite[Chapter 9]{bagr} and the references
therein. In this context, the diffraction spectrum of a point set
$\Lambda$ is defined through the Fourier transform of an autocorrelation 
measure $\gamma_\Lam$, which is associated to the set $\Lambda$ by a
certain limiting procedure. 
\par
In this section we first recall Hof's notion of diffraction, and then
apply our previous results to analyze  diffraction spectra 
of Delone sets with finite local complexity.

\subsection{} \label{9.1}
Let $\Lambda\subset\R^n$ be a Delone set.
Hof proposed  to understand the diffraction by $\Lam$ using the following procedure.
For  each $R>0$, consider the measure
\[\gamma_\Lambda^R := (2R)^{-n} \sum_{\lambda,\lambda'\in \Lambda\cap [-R,R]^n}\delta_{\lambda'-\lambda}. \]
It is a finite measure on $\R^n$ which is both positive and positive-definite.
 The uniform discreteness of $\Lambda$ implies that the measures
$\gamma_\Lambda^R$ are all translation-bounded, with the constant
in \eqref{tr-bdd-def} bounded uniformly with respect to $R$. 
Hence there exists at least one
weak limit point $\gamma_\Lambda$ of the measures
$\gamma_\Lambda^R$ as $R\to\infty$. Any such limit point
$\gamma_\Lambda$ is called an \define{autocorrelation measure} of
the set $\Lambda$.
 The  measure $\gamma_\Lambda$  is also
 translation-bounded, positive and positive-definite. 
\par
The positive-definiteness of $\gamma_\Lambda$ implies that
 its Fourier transform
$\hat{\gamma}_\Lambda$ is a positive measure. It is called a
\define{diffraction measure} of $\Lambda$.  If the measure
$\hat{\gamma}_\Lambda$ is purely atomic, then 
its support $S$ is called a
\define{diffraction spectrum} of $\Lambda$, and  $\Lambda$ is
said to be a  \define{pure point diffractive set}.
\par
More generally, one can define diffraction by any 
translation-bounded measure  $\mu$ on
$\R^n$, in a similar way. Denote by $\mu_R$ the restriction of $\mu$ to the
cube $[-R,R]^n$, and define a measure $\tilde{\mu}_R$ by 
\[ \tilde{\mu}_R(E):= \overline{\mu_R(-E)}.\]
Then  the  measures 
\[\gamma_\mu^R:=(2R)^{-n} \; \mu_R * \tilde{\mu}_R\]
are uniformly translation-bounded and so 
have at least one weak limit point $\gamma_\mu$ as $R\to\infty$, and any such
limit point is called an autocorrelation measure
of $\mu$. It is again a translation-bounded, positive-definite
measure, and if $\mu$  is a positive measure then
also $\gam_\mu$ is positive.
The  diffraction measure
$\hat{\gamma}_\mu$ and the diffraction spectrum $S$
(assuming that $\hat{\gamma}_\mu$ is  purely atomic) are also
defined in a similar way. 
\par
Notice that the diffraction by a Delone set $\Lam$ described above, is included as a
special case which corresponds to diffraction by the measure 
\begin{equation}
	\label{DD.1}
	\mu = \sum_{\lambda\in\Lambda}\delta_\lambda.
\end{equation}

\subsection{} \label{9.2}
A Delone set $\Lambda \subset \R^n$ is said to be of
\define{finite local complexity} if for every $R>0$ there
are only finitely many different sets of the form
\[ (\Lambda - \lambda)\cap B_R, \quad \lambda \in \Lambda. \] 
It is easy to verify that this condition is equivalent to the requirement that 
$\Lambda-\Lambda$ is a discrete closed set. 
\par
Notice that if    $\mu$ is a  translation-bounded measure supported by a Delone set $\Lambda$ of
finite local complexity, then any autocorrelation measure $\gamma_\mu$ of $\mu$ must be supported by 
the set $\Lambda-\Lambda$. In particular, $\gamma_\mu$ is a discrete measure.
\par
Model sets are well-studied
examples of non-periodic Delone sets of finite local complexity,
with  pure point diffraction in Hof's sense. More precisely, if
$\Lambda=\Lambda(\Gamma,\Omega)$ is a model set defined by
\eqref{RP.9.0} and such that the
boundary of the ``window'' $\Omega$ is a set of Lebesgue measure
zero, then $\Lambda$ is a pure point diffractive set, with a
dense countable diffraction spectrum (see for example  \cite[Section 9.4]{bagr}).

\subsection{} 
Let $\Lambda \subset \R^n$ be a Delone set  of finite local complexity.
Assume that the diffraction spectrum $S$ is  uniformly discrete. 
Is it true that $S$ must have a periodic structure?
\par
The question was raised in  \cite[Problem 4.2(a)]{lag2}.
It follows from our previous results that the answer is positive:
\begin{thm}
	\label{DT.1}
Suppose that
\begin{enumerate-math}
\item
$\Lambda \subset \R^n$ is a
	Delone set  of finite local complexity;
\item
$\mu$ is a positive, translation-bounded measure supported by $\Lam$;
\item
$\gamma_\mu$ is  an autocorrelation measure of $\mu$; and
\item
the support $S$ of the diffraction
	measure $\hat{\gamma}_\mu$ is a uniformly discrete set. 
\end{enumerate-math}
Then $S$ is contained in a finite union of translates of a certain lattice,
and the  diffraction measure has the form \eqref{RM.3}.
\end{thm}

\begin{proof}
The 
	autocorrelation measure $\gamma_\mu$ is positive, and is supported by $\Lambda-\Lambda$. Hence the
	diffraction measure $\hat{\gamma}_\mu$ is a positive-definite measure on $\R^n$, whose
	support $S$ is a uniformly discrete set, and whose spectrum is contained in the
	discrete closed set $\Lambda-\Lambda$. \thmref{RT.5} (applied to the measure
	$\hat{\gamma}_\mu$) therefore yields that $\hat{\gamma}_\mu$
	is of the form \eqref{RM.3}. As a consequence,  $S$ 
	must be contained in a finite union of translates of a  lattice.
\end{proof}

In particular \thmref{DT.1} applies to the measure \eqref{DD.1}.
In this case the result shows that if a Delone set
$\Lambda$ of finite local complexity is pure point diffractive, 
and if the diffraction spectrum $S$ is uniformly discrete, then
$S$ is contained in a finite union of translates of a lattice,
and the  diffraction measure has the form \eqref{RM.3}.
So we obtain \thmref{RT.8}.

	\subsection{}
If $\Lambda$ is a Meyer set, then the conclusion in the previous result
remains true even if the spectrum is just  a discrete closed set, and
without the positivity of the measure:
\begin{thm}
	\label{DT.2}
Suppose that
\begin{enumerate-math}
\item
$\Lambda$ is a Meyer set in $\R^n$;
\item
$\mu$ is a translation-bounded measure supported by $\Lam$;
\item
$\gamma_\mu$ is  an autocorrelation measure of $\mu$; and
\item
the support $S$ of the diffraction
	measure $\hat{\gamma}_\mu$ is a discrete closed set. 
\end{enumerate-math}
Then the same conclusion as in \thmref{DT.1} is true.
\end{thm}

This can be deduced from either Theorem \ref{RT.5} or \ref{RT.7},
using the fact that the auto\-correlation measure $\gamma_\mu$ is a positive-definite measure,
supported by the Meyer set $\Lambda-\Lambda$.

\subsection*{Remarks} 
{\bf 1.}\;\;Similarly, one can prove a dichotomy result for the diffraction spectrum
of a measure $\mu$  supported by a Meyer set: either
the spectrum
is uniformly discrete, or it has a relatively dense set of accumulation points
(using Theorem \ref{AT.5} or \ref{CT.1}).
\par
{\bf 2.}\;\;In the latter case, the spectrum $S$ need not be dense in any ball.
One can verify that the measure constructed in the proof of \thmref{ET.1}
is an autocorrelation of another measure whose support is also a Meyer set.

\section{Non pure point spectrum}

\subsection{}
In crystallography it is often interesting to consider also discrete measures
$\mu$, whose Fourier transform $\hat{\mu}$ is a measure which has both a pure point
component and a continuous one. The pure point component is often referred to as ``Bragg
peaks'', while the continuous component is called ``diffuse background''.
\par
 Let $\mu$ be a
(slowly increasing) measure on $\R^n$ with discrete support $\Lambda$: 
\begin{equation}
	\label{FE.1}
	\mu = \sum_{\lambda\in \Lambda} \mu(\lambda)\delta_\lambda, \quad \mu(\lambda)\ne 0.
\end{equation}
Assume that $\hat{\mu}$ is also a slowly increasing measure, and consider its decomposition 
\begin{equation}
	\label{FE.2.2}
 \hat{\mu} = \hat{\mu}_d+\hat{\mu}_c 
\end{equation}
into a sum of a pure point measure
\begin{equation}
	\label{FE.2}
	\hat{\mu}_d = \sum_{s\in S}\hat{\mu}_d(s)\delta_s, \quad \hat{\mu}_d(s)\ne 0,
\end{equation}
and a continuous measure $\hat{\mu}_c$. The set $S$ is the support of the discrete part $\hat{\mu}_d$.
\par
We can extend our previous results to this more general situation,
using the following result (\thmref{RT.10}): If
 the support $\Lambda$ is uniformly discrete, then $\hat{\mu}_d$ 
 is the Fourier transform of another  measure
	$\mu'$, whose support $\Lambda'$ is also a uniformly discrete set.

\subsection*{Remark} 
We will see from the proof that the new measure $\mu'$ is a weak limit of 
translates of $\mu$. Hence, in particular, the following is true:
\begin{enumerate-math}
\item \label{FT.1.i} If $\mu$ is a positive measure, then also $\mu'$ is positive;
\item \label{FT.1.ii}
 If $\Lambda-\Lambda$ is a discrete closed set, 
	then also $\Lambda'-\Lambda'$ is a discrete closed set; 
\item \label{FT.1.iii}
 If $\Lambda-\Lambda$ is uniformly discrete,
	then also $\Lambda'-\Lambda'$ is uniformly discrete.
\end{enumerate-math}
Property \ref{FT.1.i} is obvious.
Properties \ref{FT.1.ii} and \ref{FT.1.iii} follow  from the
 fact that  $\Lambda'-\Lambda'$ must be contained in the closure of
the  set $\Lambda-\Lambda$.

\subsection{}
First we give the proof of \thmref{RT.10}. We will use the following lemmas:
\begin{lem}
	\label{FL.1}
	Let $\nu$ be a finite measure on $\R^n$. Then 
	\[ \lim_{R\to\infty} (2R)^{-n}\int_{[-R,R]^n}\left| \hat{\nu}(t)\right|^2dt
	= \sum_{a}\left|\nu (\{a\} )\right|^2, \]
	where $a$ goes through all the atoms of the measure $\nu$. 
\end{lem} 
This is the well-known Wiener's lemma in $\R^n$. 
\begin{lem}
	\label{FL.2}
	Let $\nu$ be a (slowly increasing) continuous measure on $\R^n$. Then there exist
	vectors $\omega_k\in \R^n$ $(k\ge 1)$ such that the measures 
	\begin{equation}
		\label{FL.2.1}
		\nu_k(x):=e^{-2\pi i\dotprod{\omega_k}{x}} \,\nu(x)
	\end{equation}
	tend to zero as $k\to\infty$ in the space of temperate distributions. 
\end{lem}
\begin{proof}
	Let $\{\varphi_j\}$, $j\geq 1$, be a sequence of functions dense in the Schwartz
	space. Define 
	\[ \Phi_k(t) := \sum_{j=1}^{k}\left| \int \varphi_j(x)e^{-2\pi i
	\dotprod{t}{x}}d\nu(x)\right|^2, \quad t\in \R^n. \]
	For each $j$, the measure $\varphi_j\cdot \nu$ is finite and continuous, hence
	by \lemref{FL.1} we have
	\[ \lim_{R\to\infty}(2R)^{-n}\int_{[-R,R]^n} \Phi_k(t)dt = 0.  \]
	This implies that for each $k$ one can find $\omega_k\in\R^n$ such that $\Phi_k(\omega_k)<
	1/k^2$.
\par
 Now consider the measure $\nu_k$ defined by \eqref{FL.2.1}. We have
	\[ \left|\dotprod{\nu_k}{\varphi_j}\right| < \frac{1}{k} \quad (k\ge j) \]
	and hence $\dotprod{\nu_k}{\varphi_j}\to 0$ as $k\to\infty$, for each $j$.
Since $\nu$ is a slowly increasing measure, the sequence $\nu_k$ is uniformly bounded
  in the space of temperate distributions. Since the $\varphi_j$ are dense in the
Schwartz space, we can conclude that $\dotprod{\nu_k}{\varphi}\to 0$ as $k\to\infty$,
for every Schwartz function $\varphi$. This proves the lemma.
\end{proof}
\begin{proof}[Proof of \thmref{RT.10}]
	Use \lemref{FL.2} to find vectors $\omega_k$ such that 
	\[ e^{-2\pi i \dotprod{\omega_k}{x}} \,\hat{\mu}_c(x)\to 0, \quad k\to \infty \]
	in the space of temperate distributions. By taking a subsequence if necessary we
	can also assume  that the sequence $e^{2\pi i \dotprod{\omega_k}{s}}$ has a limit as
	$k\to\infty$, for each $s\in S$. Let $\{k_j\}$ be a sufficiently fast increasing sequence
	such that 
	\begin{equation}
		\label{FT.1.0}
	 e^{2\pi i \dotprod{\omega_j-\omega_{k_j}}{x}} \, \hat{\mu}_c(x)\to 0, \quad
	j\to\infty. 
	\end{equation}
	Since the exponential in \eqref{FT.1.0} tends to $1$ on $S$, 	it follows that the measure
	\begin{equation}
		\label{FT.1.1}
		e^{2\pi i \dotprod{\omega_j-\omega_{k_j}}{x}} \,\hat{\mu}(x) 
	\end{equation}
	tends to $\hat{\mu}_d$ as $j\to \infty$. 
The measure in \eqref{FT.1.1} is the
	Fourier transform of the measure 
	\[ \mu_j(t):= \mu(t+\omega_j-\omega_{k_j}). \]
	Therefore, $\mu_j$ tends as $j\to\infty$ to a certain distribution $\mu'$, whose
	Fourier transform is $\hat{\mu}_d$. By \lemref{AL.4} the measure $\mu$ is
	translation-bounded, and therefore also $\mu'$ is a translation-bounded measure.
	The measure $\mu_j$ is supported by the set
	\[ \Lambda_j := \Lambda-\omega_j+\omega_{k_j}, \]
	which is  uniformly discrete  with $d(\Lambda_j)=d(\Lambda)$. Hence the support $\Lambda'$
	of the measure $\mu'$ must satisfy $d(\Lambda')\ge d(\Lambda)$, so $\Lam'$ is also a uniformly
	discrete set. 
\end{proof}
\subsection{}
\thmref{RT.10} allows to extend our previous results to the case when the measure
$\hat{\mu}$ has also a continuous component. For example, we have:
\begin{thm}
	\label{FT.2}
	Let $\mu$ be a measure on $\R^n$ satisfying \eqref{FE.1}--\eqref{FE.2}. Assume
	that $\Lambda$ is uniformly discrete, $S$ is discrete and closed, and at least one
	of the following additional conditions is satisfied: 
	\begin{enumerate-math}
	\item \label{FT.2.i} $\mu$ is positive, and $S$ is uniformly discrete; 
	\item \label{FT.2.ii} $\hat{\mu}_d$ is positive; 
	\item \label{FT.2.iiii} $n=1$, and $S$ satisfies condition \eqref{RM.4};
	\item \label{FT.2.iv} $\Lambda-\Lambda$ is uniformly discrete. 
	\end{enumerate-math}
	Then $S$ is a uniformly discrete set, contained in a finite union of translates of
	a lattice, and the measure $\hat{\mu}_d$ has the form \eqref{RM.3}.
\end{thm}
To prove this we consider the measure $\mu'$ given by \thmref{RT.10}, and apply to this
measure one of Theorems \ref{RT.0}, \ref{RT.5}, \ref{RT.6} or \ref{RT.7}, according to which
one of the conditions \ref{FT.2.i}--\ref{FT.2.iv} in \thmref{FT.2} is satisfied. 
\par
In a similar way, one can extend the dichotomy results given in Theorems \ref{AT.5} or \ref{CT.1}.
\subsection{}
The same applies  to Hof's diffraction by  measures supported on  Meyer sets:
\begin{thm}
	\label{FT.3}
	Let $\mu$ be a translation-bounded measure on $\R^n$, 
	supported by a  Meyer set $\Lambda$. Let $\gamma_\mu$ be an autocorrelation measure of
	$\mu$, and denote by $S$  the support of the discrete part of the diffraction measure $\hat{\gamma}_\mu$. Then, either
	\begin{enumerate-math}
	\item $S$ is a uniformly discrete set, contained in a finite union of translates
		of a lattice, and the discrete part of $\hat{\gamma}_\mu$ has the
		form \eqref{RM.3}; or 
	\item $S$ is not a discrete closed set, and moreover the set of accumulation
		points of $S$ is relatively dense. 
	\end{enumerate-math}
\end{thm}
To prove this one can first apply \thmref{RT.10} to the measure
$\gamma_\mu$ which is supported by the Meyer set $\Lam-\Lam$,
and then use Theorem \ref{RT.0} and Theorem \ref{AT.5} or \ref{CT.1}.

\subsection*{Remark} 
An alternative proof of \thmref{FT.3}, which does not rely
on \thmref{RT.10}, can be given using the following:
\begin{lem}
	\label{FL.4}
	Let $\nu$ be a translation-bounded measure on $\R^n$, and assume that
$\ft\nu$ is a slowly increasing measure. Then $\nu$ has a unique autocorrelation measure
$\gamma_\nu$, and the diffraction measure $\hat{\gamma}_\nu$ is a pure point
measure given by
\[
\ft{\gamma}_\nu 
	= \sum_{a}\left|\ft\nu (\{a\} )\right|^2 \delta_a, \]
	where $a$ goes through all the atoms of the measure $\ft\nu$. 
\end{lem}
To prove \thmref{FT.3} using this lemma,  let $\nu := \gamma_\mu$.
Then the autocorrelation measure $\gamma_\nu$  is a discrete measure, supported by the Meyer 
set $\Lam+\Lam-\Lam-\Lam$, and 
by \lemref{FL.4} the measure $\hat{\gamma}_\nu$
is a pure point measure, with the same support as the discrete
part of $\ft{\gam}_\mu$. So our previous results can be applied
to the measure $\gamma_\nu$.
\par
We omit the proof of \lemref{FL.4}.

% ====================================================================  

\section{Remarks. Open problems}

\subsection{} 
 Very recently, Y.~Meyer  has  found \cite{mey4} an interesting version of 
\thmref{RT.4}.      Namely, he  constructed  measures $\mu$
  whose supports and spectra are
       discrete closed sets,  which can be described by simple effective  formulas. 
       He also proved   that the parameters of  this construction can be chosen so 
that both $\Lam$ and $S$ are rationally independent  sets.  This is  a  stronger ``non-periodicity''
      condition  than  in \thmref{RT.4}.  However,  such a  measure cannot be translation-bounded,
      see \cite[Lemma 5]{mey4}.
    
      It should be mentioned  that  the last paper  contains some other examples of
      measures with discrete closed  supports and spectra.   See also \cite{kol}.
      All these examples,  in one way or another,  are based  on the classical Poisson
 summation formula. 
Question: can one construct an example  which in no way is based
on Poisson's formula?

\subsection{}
We mention some problems which are left open. 
\par
1.~The first one  concerns the positivity assumption in
\thmref{RT.0} in several dimensions.
Let $\mu$ be a measure on $\R^n$, $n>1$, 
satisfying \eqref{RM.1} and \eqref{RM.2}. Assume that $\Lambda, S$ are both
uniformly discrete sets. Is it true that $\Lambda$ can be covered
by a finite union of translates of several, not necessarily
commensurate, lattices?
(an example in \cite{fav} shows that $\Lam$ need not be
contained in a finite union of translates of a \emph{single} lattice).
\par
2.~A second problem concerns the positive-definiteness assumption
in \thmref{RT.5}, even in  dimension one:
Let $\mu$ be a measure on $\R$, with uniformly discrete support
$\Lam$ and discrete closed spectrum $S$. Does it follow that
$S$ must be also uniformly discrete?
\par
3.~The following question is also open: can one get a \define{positive} measure in \thmref{RT.4}\,?

\subsection{}
Our approach to prove \thmref{RT.0} (see \cite{lo2}) involved a combination
of analytic and discrete combinatorial considerations. 
In the latter part, a conclusion about the arithmetic structure
of a set $\Lambda\subset \R^n$ was derived from information on
discreteness of $\Lambda-\Lambda$. In that point we relied on results
which go back to Meyer \cite{mey1}. 
\par
In this context, it is worth to mention Freiman's theorem \cite{freim},
which states that a finite set
$A$ such that $\#(A+A)\le K  \# A$, must be contained in a
``generalized arithmetic progression'' whose dimension and size
are controlled in terms of the constant $K$. It might be
interesting to see whether it can also be used for a 
proof of \thmref{RT.0}.

% ========================================================================

\end{document}